\documentclass[10pt, a4paper]{amsart}

\usepackage{amsmath, amssymb}
\usepackage{pdfsync}

\numberwithin{equation}{section}
\newtheorem{theorem}{Theorem}[section]

\newtheorem{proposition}[theorem]{Proposition}
\newtheorem{lemma}[theorem]{Lemma}
\theoremstyle{remark}
\newtheorem{remark}{Remark}[section]

\theoremstyle{definition}

\newcommand{\R}{{\mathbb{R}}}

\newcommand{\C}{{\mathbb{C}}}

\begin{document}

\title%
[Carleman estimates and waveguides]%
{Carleman estimates and necessary conditions for the existence of waveguides}

\author{Luis Escauriaza}
\address{Luis Escauriaza: Universidad del Pa\'is Vasco, Departamento de
Matem\'aticas, Apartado 644, 48080, Bilbao, Spain}
\email{luis.escauriaza@ehu.es}

\author{Luca Fanelli}
\address{Luca Fanelli: Universidad del Pa\'is Vasco, Departamento de
Matem\'aticas, Apartado 644, 48080, Bilbao, Spain}
\email{luca.fanelli@ehu.es}

\author{Luis Vega}
\address{Luis Vega: Universidad del Pa\'is Vasco, Departamento de
Matem\'aticas, Apartado 644, 48080, Bilbao, Spain}
\email{luis.vega@ehu.es}

\begin{abstract}
 We study via Carleman estimates the sharpest possible exponential decay for {\it waveguide} solutions to the Laplace equation
 \begin{equation*}
   (\partial^2_t+\triangle)u=Vu+W\cdot(\partial_t,\nabla)u,
 \end{equation*}
and find a necessary quantitative condition on the exponential decay in the spatial-variable of nonzero waveguides solutions which depends on the size of $V$ and $W$ at infinity.
\end{abstract}

\date{\today}    

\subjclass[2010]{AMS Subject Classification: 35B99, 35J05, 35J15}
\keywords{
Carleman estimates, waveguides, unique continuation} 
\maketitle


\section{Introduction}\label{sec:intro}

Consider the Laplace equation

\begin{equation}\label{eq:waveeq}
  \partial_t^2u+\triangle u= V(t,x)u,
\end{equation}
where $u=u(t,x):\R\times\R^{n}\to\C$, $V=V(t,x):\R\times\R^{n}\to\C$, $n\geq1$, with $\nabla=\nabla_x$ and $\triangle=\nabla\cdot\nabla u$. Assume that $V=V(x)$ has a positive eigenvalue $\lambda$ with corresponding eigenfunction $Q=Q(x)$, so that
\begin{equation}\label{eq:stationary}
  \triangle Q(x)-\lambda Q(x)=V(x)Q(x),\  \text{in}\ \R^n.
\end{equation}
It is known that if $V$ decays fast enough at infinity, then $Q$ has  exponential decay; i.e.
\begin{equation}\label{eq:stationarydec}
  Q(x)\leq Ce^{-\sqrt\lambda|x|},\ \text{for some}\ C>0
\end{equation}
and then
\begin{equation}\label{eq:waveguide}
  u_1(t,x)=\cos{\left(t\sqrt\lambda\right)}Q(x),\
  u_2(t,x)=\sin{\left(t\sqrt\lambda\right)}Q(x)
\end{equation}
are solutions of \eqref{eq:waveeq} in $\R^{n+1}$. Such type of solutions are called {\it waveguides} solutions of \eqref{eq:waveeq}.

Notice that {\it waveguide} solutions do not decay in the time-variable and that
\begin{equation}\label{eq:waveguidedec}
  \sup_{t\in\R}\int e^{\beta|x|}|u|^2\,dx<\infty, \ \text{only for}\ \beta<2\sqrt\lambda.
\end{equation}
It is surprising that {\it waveguide} solutions to 
\eqref{eq:waveeq} also exist in the nonlinear case, when $V=V(t,x)=|u|^p$. This was first remarked by Busca and Felmer \cite{BF} and by Dancer \cite{D} in the case $n=1$ (they study respectively the existence of odd or even solutions with respect to the time-variable) and later studied by other authors (See \cite{DKPW} and the references therein).

We are interested in finding conditions like \eqref{eq:waveguidedec} in order to describe the sharpest possible decay of a waveguide solution, when $V$ decays sufficiently fast with respect to the spatial variables at infinity. More generally, condition
\eqref{eq:waveguidedec} should be replaced by
\begin{equation}\label{eq:waveguidedecgen}
  \sup_{t\in\R}\int e^{\beta|x|^p}|u|^2\,dx<\infty,
  \qquad
  \beta<2\sqrt\lambda,
\end{equation}
where the power $p$ depends on the decay-rate of $V$.
These questions are related to unique continuation problems at infinity which arise in different contests and with different applications. 

In the last few years, some efforts have been made in this direction. Let us focus our attention on potentials $V$ verifying, $|V|\lesssim |x|^{-\alpha}$, as $|x|\to\infty$, with $\alpha\geq0$. In the stationary case, 
$u=u(x)$, we first mention Froese, Herbst, T. Hoffman-Ostenhof, and M. Hoffman-Ostenhof, who proved in \cite{FH} that if 
$\alpha\geq\frac12$, the best possible decay in \eqref{eq:waveguidedecgen} is $p=1$, and in addition that the result is sharp.
Later, Meshkov proved \cite{M} that if $\alpha=0$ the best possible condition is obtained by 
\eqref{eq:waveguidedecgen} with $p=\frac43$; and later, Cruz-Sampedro \cite{C} generalized the result to the case $0\leq\alpha<\frac12$, and finding the relation, 
$p=\frac{4-2\alpha}{3}$. In the last two papers, suitable Carleman estimates are  fundamental tools. Moreover, we remark that the examples of potentials constructed by Meshkov and Cruz-Sampedro are complex-valued.
More recently, the analogous problem for the Schr\"odinger evolution equation has been studied in \cite{EKPV1}, \cite{EKPV2}. In these works, the authors are also interested on the sharpest possible exponential-decay of solutions to the nonlinear Schr\"odinger equation which blow up in finite time. Also in this work, a suitable Carleman estimate, based on the convexity properties of solutions to \eqref{eq:waveeq} plays a crucial role.

In analogy with the results mentioned above, we state now our main result, in which we also consider first order perturbations.
\begin{theorem}\label{thm:wave}
Let $V=V(t,x):\R\times\R^n\to\C$ and $W=W(t,x):\R\times\R^n\to\C\times\C^n$ satisfy \footnote{Hereafter $\rho + $ will mean any real number strictly bigger than $\rho$. Similarly for $\rho - $.}
  \begin{align}\label{eq:assV}
   V\in L^\infty(\R\times\R^n),
   &
   \qquad
   \sup_{t\in\R}|V(t,x)|\leq\frac C{(1+|x|^2)^{\frac14+}},
   \\
   W\in L^\infty(\R\times\R^n),
   &
   \qquad
   \sup_{t\in\R}|W(t,x)|\leq \frac C{(1+|x|^2)^{\frac14+}},
   \label{eq:assW}
  \end{align}
  for some $C>0$. Then, there exists $\lambda_0=\lambda_0(C,n)$ such that if $u=u(t,x)$ in $C^1(\R, H^1(\R^n))\cap H^2_{loc}(\R\times\R^n)$ verifies the differential inequality
  \begin{equation}\label{eq:wave}
    \left|\left(\partial_t^2+\triangle\right)u\right|\leq|Vu|+\left|W\cdot(\partial_t,\nabla)u\right|,
  \end{equation}
in $\R^{n+1}$ and
  \begin{equation}\label{eq:assumptionH}
  \sup_{t\in\R}\int e^{4\lambda|x|}u^2\,dx<\infty,
  \end{equation}
  for some $\lambda >\lambda_0$, then $u\equiv0$.
\end{theorem}
\begin{remark}[Decay]\label{rem:1}
  Notice that the decay rate $1/{(1+|x|^2)^{\frac14+}}$ in \eqref{eq:assV} and \eqref{eq:assW} is in analogy with the one required in Theorem 2 in \cite{EKPV2} for the Schr\"odinger evolution case and fits with the results in \cite{C} and \cite{FH}.
\end{remark}
\begin{remark}[Regularity]\label{rem:2}
  The regularity assumption $u\in H^2_{loc}(\R\times\R^n)$ is required in order to justify both the left and right-hand sides of the inequality \eqref{eq:wave}, while the assumption 
  $u\in C^1(\R, H^1(\R^n))$ is needed to justify the integration by parts in Proposition \ref{prop:1} below. In fact, the result applies to all the distibutional solutions to
  \begin{equation*}
    \partial_t^2u+\triangle u= Vu+W\cdot(\partial_t,\nabla)u,
  \end{equation*}
 verifying \eqref{eq:assumptionH}.
\end{remark}
\begin{remark}[Repulsive perturbations]\label{rem:3}
A more general version of the previous theorem, including repulsive zero-order repulsive perturbations can also be proved in analogy with Theorem 2 in \cite{EKPV2}. More precisely, we can prove the same result for solutions of the inequality
 \begin{equation}\label{eq:wavegen}
    \left|\left(\partial_t^2+\triangle+V_2\right)u\right|\leq|Vu|+\left|W\cdot(\partial_t,\nabla)u\right|,
  \end{equation}
with $V$ and $W$ as in Theorem \ref{thm:wave} and where the potential $V_2$ verifies 
  \begin{equation}\label{eq:assV2}
  V_2\in L^\infty(\R\times\R^n) \ \text{and}\ 
   \sup_{t\in\R}\left(\partial_rV_2(t,x)\right)_-=\frac C{(1+|x|^2)^{\frac12+}},
  \end{equation}
  where $\partial_r$ denotes the radial derivative and $\left(\partial_rV_2\right)_-$ its negative part. See Appendix \ref{sec:app1} for the details.
\end{remark}
An analogous result when $V$ and $W$ 
decay more slowly at infinity is the following.
\begin{theorem}\label{thm:wave2}
  Let $V=V(t,x):\R\times\R^n\to\C$, $W=W(t,x):\R\times\R^n\to\C\times\C^n$ satisfy
  \begin{align}\label{eq:assV00}
   & V\in L^\infty(\R\times\R^n),\quad
   \sup_{t\in\R}|V(t,x)|\leq \frac C{(1+|x|^2)^{\frac\alpha2}},
   \\
   & W\in L^\infty(\R\times\R^n),\quad 
   \sup_{t\in\R}|W(t,x)|\leq \frac C{(1+|x|^2)^{\frac{1+\alpha}{6}}},
   \label{eq:assW00}
  \end{align}
  for some $C>0$ and $0\leq\alpha<\frac12$. Then, there is $\lambda_0=\lambda_0(C,n,\alpha)$ such that if $u=u(t,x)$ in $C^1(\R;H^1(\R^n))\cap H^2_{loc}(\R\times\R^n)$ verifies
  \begin{equation}\label{eq:wave00}
    \left|\left(\partial_t^2+\triangle\right)u\right|\leq|Vu|+\left|W\cdot(\partial_t,\nabla)u\right|,
  \end{equation}
in $\R^{n+1}$,
  \begin{equation}\label{eq:assumptionH00}
  \sup_{t\in\R}\int e^{4\lambda|x|^p}u^2\,dx<\infty,
  \end{equation}
  with $  p=\frac{4-2\alpha}{3}$ and for some $\lambda>\lambda_0$, then $u\equiv0$.
\end{theorem}
\begin{remark}
 Observe the analogy with Theorem 1 in \cite{EKPV2} and notice that the decay conditions \eqref{eq:assV00}, \eqref{eq:assW00} fit with the results by Meshkov \cite{M} for $\alpha=0$ and by Cruz-Sampedro \cite{C} for the rest of the range $\alpha\in[0,\frac12)$ in the stationary case.
\end{remark}
The proof of Theorem \ref{thm:wave2} is sketched in Appendix \ref{sec:app2} because it is analogous to the one for Theorem \ref{thm:wave}. In this case one can also consider repulsive zero-order perturbations to the equation \eqref{eq:wave00}, as in Remark \ref{rem:3} above. The details are omitted. The rest of the paper is devoted to the proof of Theorem \ref{thm:wave}.

\section{Monotonicity Lemma}\label{sec:carleman}

We prove here an abstract monotonicity result, which will be used later in order to obtain suitable Carleman estimates.

 In the following, the symbol $\mathcal S$ denotes a symmetric operator on $L^2(\R^n)$, and $\mathcal A$ a skew-symmetric one, both of them independent on $t$. Moreover, the brackets $[\cdot,\cdot]$ stay for the commutator of two operators, and given two functions $f=f(t,x), g=g(t,x)$, we denote 
 $\langle f,g\rangle=\int f\overline g\,dx$, the hermitian product in the $x$-variable. 
  
We can now state the main result of this section.

\begin{proposition}[Monotonicity Lemma]\label{prop:1}
  Let $\mathcal S$ be a symmetric operator and $\mathcal A$ be a skew-symmetric one, both independent on time.  Then, for sufficiently regular functions $f=f(t,x):\R^{1+n}\to\C$, $\psi=\psi(x):\R^{n}\to\R$, and for any $T_0,T_1\in\R$, the following estimate holds:  
  \begin{align}\label{eq:carleman1}
    &\frac12\int_{T_0}^{T_1}\int[\mathcal S,\mathcal A]f\,\overline f\,dx\,dt+
    \int_{T_0}^{T_1}\int\psi\left|\partial_tf\right|^2\,dx\,dt+
    \frac12\int_{T_0}^{T_1}\int|\mathcal Af|^2\,dx\,dt
    \\
    & +
    \Re\int_{T_0}^{T_1}\int(\mathcal S+\mathcal A)f\,\psi\overline f\,dx\,dt
    \nonumber
    \\
    &
    \leq
    \int_{T_0}^{T_1}\int\left|(\partial_t^2-(\mathcal S+\mathcal A))f\right|^2\,dx\,dt+
    \frac12\int_{T_0}^{T_1}\int\psi^2|f|^2\,dx\,dt
    \nonumber
    \\
    & +
    \left.\left|\langle \mathcal Af,\partial_tf\rangle+\langle\psi f,\partial_tf\rangle\right|\,
    \right\vert_{t=T_0}^{t=T_1}.
    \nonumber
  \end{align}
\end{proposition}

\begin{proof}
  Notice that $\Re\langle \mathcal Af,f\rangle=0$, since $\mathcal A$ is skew-symmetric; analogously, we have that 
  $\Re\langle \mathcal Af,\mathcal Sf\rangle=\frac12\langle[\mathcal S,\mathcal A]f,f\rangle$. Hence, the explicit computation of  
  \begin{equation}\label{eq:frontiera}
    \frac{d}{dt}\Re\left(\langle \mathcal Af,\partial_tf\rangle
    +\langle\psi f,\partial_tf\rangle\right)
  \end{equation}
  gives the following identity:
  \begin{align}\label{eq:1}
    &
    \frac12\langle[\mathcal S,\mathcal A]f,f\rangle+
    \langle\psi\partial_tf,\partial_tf\rangle+
    \langle \mathcal Af,\mathcal Af\rangle+
    \Re\langle\psi f,(\mathcal S+\mathcal A)f\rangle
    \\
    & 
    =
    -\Re\langle \mathcal Af,(\partial_t^2-(\mathcal S+\mathcal A))f\rangle-
    \Re\langle\psi f,(\partial_t^2-(\mathcal S+\mathcal A))f\rangle
    +\frac{d}{dt}\Re\left(\langle \mathcal Af,\partial_tf\rangle
    +\langle\psi f,\partial_tf\rangle\right).
    \nonumber
  \end{align}
  By Cauchy-Schwartz we obtain
  \begin{align}\label{eq:4}
    & 
    \left|\int_{T_0}^{T_1}\langle \mathcal Af,(\partial_t^2-(\mathcal S+\mathcal A))f\rangle\,dt\right|
    \\
    & \ \ 
    \leq\frac12\int_{T_0}^{T_1}\int|Af|^2\,dx\,dt
    +\frac12\int_{T_0}^{T_1}\int\left|(\partial_t^2-(\mathcal S+\mathcal A))f\right|^2\,dx\,dt.
    \nonumber
  \end{align}
  Analogously, we can estimate
  \begin{align}\label{eq:5}
      & \left|\int_{T_0}^{T_1}\langle\psi f,(\partial_t^2-(\mathcal S+\mathcal A))f\rangle\,dt\right|
      \\
      &\ \ 
     \leq
     \frac12\int_{T_0}^{T_1}\int\psi^2|f|^2\,dx\,dt
    +\frac12\int_{T_0}^{T_1}\int\left|(\partial_t^2-(\mathcal S+\mathcal A))f\right|^2\,dx\,dt.  
    \nonumber
    \end{align}
  The thesis now follows, after integrating identity \eqref{eq:1}, by estimates \eqref{eq:4}, \eqref{eq:5}.
\end{proof}
\begin{remark}
  We point out that \eqref{eq:1} is in fact a convexity argument. Indeed, notice that by denoting
 \begin{equation}\label{eq:H}
   H(t)=\langle\partial_tf,\partial_tf\rangle-\langle \mathcal Sf,f\rangle+\langle\psi f,f\rangle,
 \end{equation}
and differentiating it with respect to time, we get
 \begin{equation}\label{eq:Hdot}
   \frac{d}{dt}H(t)=
   2\Re\langle \mathcal Af,\partial_tf\rangle+2\Re\langle\psi f,\partial_tf\rangle
   +2\Re\langle\left(\partial_t^2-(\mathcal S+\mathcal A)\right)f,\partial_tf\rangle,
 \end{equation}
for $f$ regular enough. As a consequence, \eqref{eq:frontiera} is in fact a second derivative of $H(t)$ when $f$ is a solution to $(\partial_t^2-(\mathcal S+\mathcal A))f=0$. When this is not the case we treat, $\partial_t^2f-(\mathcal S+\mathcal A)f$, as an error term.
\end{remark}

\section{Proof of Theorem \ref{thm:wave}}\label{sec:proof1}
We now pass to the proof of Theorem  \ref{thm:wave} which is divided into some steps.

\subsection{Preliminaries}
Let $u$ be a solution of \eqref{eq:wave}. Denote by 
\begin{equation}\label{eq:f}
  f(t,x)=e^{\lambda\phi(x)}u(t,x),
  \qquad
  \lambda>0,
\end{equation}
where $\phi=\phi(x):\R^n\to\R$ is a sufficiently regular function to be chosen later.
It turns out that $f$ satisfies the following inequality
\begin{equation}\label{eq:feq}
\left|\left(\partial_t^2-(\mathcal S+\mathcal A)\right)f\right|\leq|Vf|+
\lambda\left|\left(\widetilde{W}\cdot\nabla\phi\right)f\right|+\left|W\cdot(\partial_t,\nabla)f\right|,
\end{equation}
where 
\begin{equation}\label{eq:wtilde}
  \widetilde W=\widetilde W(t,x):\R\times\R^n\to\R^n,
  \qquad
  \widetilde W=\left(W^1,\dots,W^{n}\right),
\end{equation}
\begin{equation}\label{eq:SA}
\mathcal S=-\triangle-\lambda^2|\nabla\phi|^2,
\qquad
\mathcal A=\lambda\left(2\nabla\phi\cdot\nabla+\triangle\phi\right).
\end{equation}
Notice that $\mathcal S$ is a symmetric operator, while $\mathcal A$ is a skew-symmetric one. Moreover, we can compute explicitly the commutator
\begin{equation}\label{eq:commutator}
  [\mathcal S,\mathcal A]=-\lambda\left(4\nabla\cdot D^2\phi\nabla-4\lambda^2\nabla\phi D^2\phi\nabla\phi
  +\triangle^2\phi\right),
\end{equation}
where $D^2\phi$ is the Hessian of $\phi$ and $\triangle^2\phi=\triangle(\triangle\phi)$ is its bi-Laplacian.

In addition, for a sufficiently regular function $\psi=\psi(x):\R^n\to\R$, integrating by parts we obtain
\begin{equation}\label{eq:S+A}
  \Re\int(\mathcal S+\mathcal A)f\,\psi\overline f\,dx=
  \int\psi|\nabla f|^2-\frac12\int(\triangle\psi)|f|^2-
  \lambda^2\int|\nabla\phi|^2\psi|f|^2-
  \lambda\int|f|^2\nabla\phi\cdot\nabla\psi.
\end{equation}

\subsection{Choice of the multipliers $\phi,\psi$}
We now choose two explicit multipliers $\phi$ and $\psi$. Let us consider, $\phi=\phi(|x|)$, a radial function of the form
\begin{equation}\label{eq:fi}
  \phi(r)=
  \begin{cases}
    a_1r^2+a_2r^4+a_3r^6+a_4r^8
    \quad
    r\in[0,1]
    \\
    3r-\int_1^r\frac{ds}{1+\log s}+a_5
    \quad
    r>1
  \end{cases}
\end{equation}
where $r:=|x|$,
for some constants $a_1,a_2,a_3,a_4,a_5\in\R$. We have the following:
\begin{lemma}\label{lem:fi}
  There exist $a_1,a_2,a_3a_4,a_5\in\R$ such that the function $\phi$ defined in \eqref{eq:fi} is strictly convex on compact sets of $\R^n$ and satisfies
  \begin{align}\label{eq:fi2}
    &
    \phi\in\mathcal C^4(\R^n);
    \qquad
    \phi(0)=0,\ \ \ \ \phi(r)>0\ \text{for}\ r>0;
    \\
    &\exists M>0:\phi(r)\leq Mr,\ \forall r\in[0,\infty);
    \label{eq:fi3}
  \end{align}
\end{lemma}
The proof of Lemma \ref{lem:fi} can be performed constructing $\phi$ explicitly in formula 
\eqref{eq:fi}, see Appendix 7 part (b) in \cite{EKPV2}.
Notice also that, by explicit computations, we have
\begin{equation}\label{eq:fieps4}
  \phi'(r)=
  \begin{cases}
    a_1r+o(r)
    \quad
    r\in[0,1]
    \\
    3-\frac{1}{1+\log r}
    \quad
    r>1,
  \end{cases}
\end{equation}
\begin{equation}\label{eq:fieps5}
  \phi''(r)=
  \begin{cases}
    a_1+o(1)
    \quad
    r\in[0,1]
    \\
    \frac{1}{r(1+\log r)^2}
    \quad
    r>1.
  \end{cases}
\end{equation}
Now, since $\phi$ is radial, we can write
\begin{equation}\label{eq:fi6}
  \nabla\phi D^2\phi\nabla\phi=(\phi''(r))(\phi'(r))^2;
\end{equation}
Hence, by \eqref{eq:fieps4}, \eqref{eq:fieps5}, and \eqref{eq:fi6}, we obtain the following estimate
\begin{equation}\label{eq:fi8}
  \nabla\phi D^2\phi\nabla\phi >
  \begin{cases}
    Cr^2
    \quad
    r\in[0,1]
    \\
     \frac{1}{r(1+\log r)^2}
     \quad
     r>1,
  \end{cases}
\end{equation}
for some constant $C>0$.
Moreover, since $\phi$ is strictly convex, we have
\begin{equation}\label{eq:stri}
  D^2\phi(r)\geq\phi''(r)Id,
  \qquad
  r\in[0,\infty).
\end{equation}
Also notice that
\begin{equation}\label{eq:fi9}
  \triangle^2\phi\in L^\infty(\R^n),
  \qquad
  |\triangle^2\phi|\leq\frac{C}{r^3}
  \quad
  (r>1).
\end{equation}
Finally,
we define
\begin{equation}\label{eq:psi}
  \psi(r)=\delta\lambda\phi''(r),
\end{equation}
for some positive constant $\delta>0$ to be chosen small enough in the sequel, and $\lambda$ the same as in assumption \eqref{eq:assumptionH}. Notice that $\psi\geq0$. 

We now prove a fundamental lemma, in which we obtain the positivity of the left-hand side of estimate \eqref{eq:carleman1}.
\begin{lemma}\label{lem:positivity}
  Let $\mathcal S,\mathcal A$ be defined by \eqref{eq:SA}, $\phi,\psi$ as above and $f$ be a sufficiently regular function. Then the following estimate holds
  \begin{equation}\label{eq:positivity}
    \int[\mathcal S,\mathcal A]f\,\overline f\,dx+\Re\int(\mathcal S+\mathcal A)f\,\psi\overline f\,dx
    \geq
    C\lambda\int\phi''|\nabla f|^2\,dx+C\lambda^3\int|f|^2\nabla\phi D^2\phi\nabla\phi\,dx,
  \end{equation}
  for some $C>0$.
\end{lemma}
\begin{proof}
  By \eqref{eq:commutator} and \eqref{eq:S+A} we can write
  \begin{align}\label{eq:100}
    &
    \int[\mathcal S,\mathcal A]f\,\overline f\,dx+\Re\int(\mathcal S+\mathcal A)f\,\psi\overline  f\,dx
    \\
    &
    =
    4\lambda\int\nabla fD^2\phi\nabla\overline  f\,dx+
    4\lambda^3\int|f|^2\nabla\phi D^2\phi\nabla\phi\,dx
    -\lambda\int|f|^2\triangle^2\phi\,dx
    \nonumber
    \\
    & \ \ \ 
    +\int\psi|\nabla f|^2\,dx-\frac12\int|f|^2\triangle\psi\,dx-
    \lambda^2\int|f|^2\psi|\nabla\phi|^2\,dx-\lambda\int|f|^2\nabla\phi\cdot\nabla\psi\,dx
    \nonumber
    \\
    & =:
    A+B+C+D+E+F+G.
    \nonumber
  \end{align}
  Notice that $D\geq0$, so that we can neglect this term in \eqref{eq:100}.
  
  By \eqref{eq:stri} we have
  \begin{equation}\label{eq:azz1}
    A = 4\lambda\int\nabla fD^2\phi\nabla\overline  f\,dx
    \geq C\lambda\int\phi''|\nabla f|^2\,dx,
  \end{equation}
  for some $C>0$. Now notice that, by \eqref{eq:fieps4}, \eqref{eq:fi8} and \eqref{eq:psi} we easily obtain
  \begin{equation}\label{eq:1000parz}
    \lambda\nabla\phi D^2\phi\nabla\phi-\psi|\nabla\phi|^2
    \geq c\nabla\phi D^2\phi\nabla\phi,
  \end{equation}
  for some $c>0$, and
  all $x\in\R^n$, if $\delta >0$ is chosen to be small enough in \eqref{eq:psi}; as a consequence, we have
  \begin{equation}\label{eq:nuova}
    B+F\geq c\lambda^3\int|f|^2\nabla\phi D^2\phi\nabla\phi\,dx,
  \end{equation}
  for $\triangle$ sufficiently small.
 We now pass to the term $C$ in identity \eqref{eq:100}.  
 Notice that $B$ vanishes, close to the origin, by \eqref{eq:fi8}.
 Let us now write
  \begin{equation}\label{eq:130}
    C=-\lambda\int_{|x|\leq\epsilon}|f|^2\triangle^2\phi\,dx   
    -\lambda\int_{|x|\geq\epsilon}|f|^2\triangle^2\phi\,dx=:C_1+C_2.
  \end{equation}
  For the term outside of the ball, since $\triangle^2\phi$ decays at infinity faster than $1/r(1+\log r)^2$, we have by \eqref{eq:fi8} that
  \begin{equation}\label{eq:140}
    B+C_2\geq
    c\lambda^3\int_{|x|\geq\epsilon}|f|^2\nabla\phi D^2\phi\nabla\phi\,dx,
  \end{equation}
  for some $c>0$.
  For the term inside the ball, we use a localized version of the Poincar\'e inequality: take 
  $\chi\in C^\infty(\R^n)$ such that $\chi\equiv1$, for $|x|\leq\epsilon$, $supp\chi\subset\{|x\leq2\epsilon\}$, and $|\nabla\chi|\leq c/\epsilon$, to estimate
  \begin{align}\label{eq:150}
    |C_1| &
    \leq
    \lambda\|\triangle^2\phi\|_{L^\infty}\int_{|x|\leq2\epsilon}\chi^2|f|^2\,dx
    \leq
    c\epsilon^2\lambda\|\triangle^2\phi\|_{L^\infty}\int_{|x|\leq2\epsilon}
    |\nabla(\chi f)|^2\,dx
    \\
    &
    \leq
    c\epsilon^2\lambda\|\triangle^2\phi\|_{L^\infty}\int_{|x|\leq2\epsilon}
    |\nabla f|^2\,dx+c\lambda\|\triangle^2\phi\|_{L^\infty}\int_{\epsilon\leq|x|\leq2\epsilon}
    |f|^2\,dx,
    \nonumber
  \end{align}
  for some $c>0$;
  consequently, we have
  \begin{equation}\label{eq:160}
    A+B+C_1\geq
    C\lambda\int_{|x|\leq2\epsilon}\phi''|\nabla f|^2\,dx+C\lambda^3\int
    _{\epsilon\leq|x|\leq2\epsilon}|f|^2\nabla\phi D^2\phi\nabla\phi\,dx,
  \end{equation}
  for  $\epsilon>0$ sufficiently small, depending on $\lambda$, and some constant $C>0$.
  In conclusion
  \begin{equation}\label{eq:azz2}
    A+B+C\geq c(A+B),
  \end{equation}
  for some $c>0$, if $\epsilon$ is chosen to be small enough, depending on $\lambda>0$.
  In a completely analogous way we get,  thanks to \eqref{eq:fieps4} and \eqref{eq:psi}, that
  \begin{equation}\label{eq:azz2222}
  A+B+E+G\geq c(A+B),
  \end{equation}
  for some $c>0$, and this completes the proof.
\end{proof}

\subsection{Proof of Theorem \ref{thm:wave}}\label{subsec:proofwave}

Let $f$ be as in \eqref{eq:f}, with $\phi,\psi$ given by \eqref{eq:fi}, \eqref{eq:psi}.  The first step is to prove a global (in $t$ and $x$) estimate, assuming \eqref{eq:assumptionH}. 
\begin{lemma}\label{lem:step2}
  Let $\phi,\psi$ given by \eqref{eq:fi}, \eqref{eq:psi},  $\mathcal S$ and $\mathcal A$ be as in \eqref{eq:SA}, and assume \eqref{eq:assV}, \eqref{eq:assW}. Then, there exists a 
  $\lambda_0>0$ large enough, depending only $V,W$, such that, if \eqref{eq:assumptionH} holds, for 
  $\lambda>\lambda_0$, then the following estimate holds:
 \begin{equation}\label{eq:step2}
    \int_{-\infty}^{+\infty}\int(\psi\left|\partial_tf\right|^2
    +\lambda\phi''|\nabla f|^2+\lambda^3|f|^2\nabla\phi D^2\phi\nabla\phi)\,dx\,dt
    \leq C,    
  \end{equation}
  for some constant $0<C=C(\lambda)$.
\end{lemma}
\begin{proof}
By standard elliptic regularity, since $u$ satisfies \eqref{eq:wave}, and $V\in L^\infty$, \eqref{eq:assumptionH} implies that
\begin{equation}\label{eq:caciotta}
  \sup_{t\in\R}\int e^{2\lambda|x|}\left(\left|\partial_tu\right|^2+\left|\nabla u\right|^2+|u|^2\right)\,dx<\infty
\end{equation}
(see e.g. \cite{GT}, Corollary 6.3). Hence
by \eqref{eq:caciotta}, \eqref{eq:SA} and \eqref{eq:psi}, we have
\begin{equation*}
  \sup_{t\in\R}\left|\langle \mathcal Af+\psi f,\partial_tf\rangle\right|<\infty.
\end{equation*}
Now, by Proposition \ref{prop:1} and Lemma \ref{lem:positivity}  we can estimate
  \begin{align}\label{eq:control}
    &
    \int_{T_0}^{T_1}\int \psi\left|\partial_tf\right|^2\,dx\,dt+
    \lambda\int_{T_0}^{T_1}\int \phi''|\nabla f|^2\,dx\,dt
    +
    \lambda^3\int_{T_0}^{T_1}\int|f|^2\nabla\phi D^2\phi\nabla\phi\,dx\,dt
    \\
    &
    \leq
    C\int_{T_0}^{T_1}\int \left(|V|^2+\lambda^2
    \left|\widetilde W\cdot\nabla\phi\right|^2\right)|f|^2\,dx\,dt+
    +C\int_{T_0}^{T_1}\int\left|W\cdot(\partial_t,\nabla)f\right|^2\,dx\,dt
    \nonumber
    \\
    & \ \ \ \ 
    C\int_{T_0}^{T_1}\int \psi^2|f|^2\,dx\,dt
        +
    C\Re\left.\left(\langle Af,\partial_tf\rangle+\langle\psi f,\partial_tf\rangle\right)
    \right\vert_{t=T_0}^{t=T_1}
    \nonumber
    \\
    &
   \leq
    C\int_{T_0}^{T_1}\int \left(|V|^2+\lambda^2
    \left|\widetilde W\cdot\nabla\phi\right|^2\right)|f|^2\,dx\,dt+
    +C\int_{T_0}^{T_1}\int\left|W\cdot(\partial_t,\nabla)f\right|^2\,dx\,dt
    \nonumber
    \\
    & \ \ \ \ 
    C\int_{T_0}^{T_1}\int \psi^2|f|^2\,dx\,dt
        +
    C,
    \nonumber
  \end{align}
  for any $T_0,T_1>0$, and for some $C>0$.
  
  By assumption \eqref{eq:assW}, using \eqref{eq:fieps4} and \eqref{eq:fi8} we see that
  \begin{align}\label{eq:200}
    & 
    \lambda^3\int_{T_0}^{T_1}\int|f|^2\nabla\phi D^2\phi\nabla\phi\,dx\,dt
    -
    C\lambda^2\int_{T_0}^{T_1}\int
    \left|\widetilde W\cdot\nabla\phi\right|^2|f|^2\,dx\,dt
    \\
    &
    \geq
    C_1\lambda^3\int_{T_0}^{T_1}\int|f|^2\nabla\phi D^2\phi\nabla\phi\,dx\,dt,
    \nonumber
  \end{align}
  for some constant $C_1>0$ and $\lambda>\lambda_0>0$, with $\lambda_0$ large enough, depending only on $W$. 
  For the term containing $V$, we need to separate the integrals close and far away from the origin. Let us write
  \begin{equation*}
    \int_{T_0}^{T_1}\int|V|^2|f|^2\,dx\,dt=
    \int_{T_0}^{T_1}\int_{|x|\leq\epsilon}|V|^2|f|^2\,dx\,dt
    +\int_{T_0}^{T_1}\int_{|x|\geq\epsilon}|V|^2|f|^2\,dx\,dt;
  \end{equation*}
  then, arguing as in Lemma \ref{lem:positivity}, by the Poincar\'e inequality and assumption 
  \eqref{eq:assV} we easily obtain
  \begin{align}\label{eq:200200}
    &
    \lambda\int_{T_0}^{T_1}\int \phi''|\nabla f|^2\,dx\,dt
    +\lambda^3\int_{T_0}^{T_1}\int|f|^2\nabla\phi D^2\phi\nabla\phi\,dx\,dt
    -C\int_{T_0}^{T_1}\int|V|^2|f|^2\,dx\,dt
    \\
    &
    \geq
    C_2\lambda\int_{T_0}^{T_1}\int \phi''|\nabla f|^2\,dx\,dt
    +C_2\lambda^3\int_{T_0}^{T_1}\int|f|^2\nabla\phi D^2\phi\nabla\phi\,dx\,dt,
    \nonumber
  \end{align}
  for some $C_2>0$ and $\lambda>\lambda_0>0$, with $\lambda_0$ large enough, depending only on $V$. 
Analogously, by assumption \eqref{eq:assW} we can estimate
  \begin{align}\label{eq:200bis}
    &
    \int_{T_0}^{T_1}\int \psi\left|\partial_tf\right|^2\,dx\,dt+
    \lambda\int_{T_0}^{T_1}\int \phi''|\nabla f|^2\,dx\,dt
    -C\int_{T_0}^{T_1}\int\left|W\cdot(\partial_t,\nabla)f\right|^2\,dx\,dt
    \\
    &
    \geq
    C_3\int_{T_0}^{T_1}\int \psi\left|\partial_tf\right|^2\,dx\,dt+
    C_3\lambda\int_{T_0}^{T_1}\int \phi''|\nabla f|^2\,dx\,dt,
    \nonumber
  \end{align}
for some constant $C_3>0$ and  $\lambda>\lambda_0>0$, with $\lambda_0$ large enough, depending only on $W$.
  Finally, write 
  \begin{equation*}
    \int_{T_0}^{T_1}\int\psi^2|f|^2\,dx\,dt
    =
    \int_{T_0}^{T_1}\int_{|x|\leq\epsilon}\psi^2|f|^2\,dx\,dt+
    \int_{T_0}^{T_1}\int_{|x|\geq\epsilon}\psi^2|f|^2\,dx\,dt,
  \end{equation*}
  and choose $\epsilon=\lambda^{-\frac12}$.
  By \eqref{eq:fi8}, \eqref{eq:psi} and the Poincar\'e inequality as above we obtain, for $\delta$ sufficiently small in \eqref{eq:psi},
  \begin{align}\label{eq:300}
    & 
    \lambda\int_{T_0}^{T_1}\int \phi''|\nabla f|^2\,dx\,dt+
    \lambda^3\int_{T_0}^{T_1}\int|f|^2\nabla\phi D^2\phi\nabla\phi\,dx\,dt
    -
    C\int_{T_0}^{T_1}\int\psi^2|f|^2\,dx\,dt
    \\
    &
    \geq
    C_4\lambda\int_{T_0}^{T_1}\int \phi''|\nabla f|^2\,dx\,dt+
    C_4\lambda^3\int_{T_0}^{T_1}\int|f|^2\nabla\phi D^2\phi\nabla\phi\,dx\,dt,
    \nonumber
  \end{align}
  for some $C_4>0$.
 The proof now follows by \eqref{eq:control}, \eqref{eq:200}, \eqref{eq:200200}, \eqref{eq:200bis} and \eqref{eq:300},
  since $T_0,T_1$ are arbitrary.
\end{proof}

Lemma \ref{lem:step2} shows that, if \eqref{eq:assumptionH} holds for $\lambda$ large enough, depending on $V,W$, then we get the stronger global estimate \eqref{eq:step2}. Moreover, 
\eqref{eq:step2} implies that there exists a sequence of times $T_j$, $j\in\mathbb Z$, with $T_j\to\pm\infty$, as $j\to\pm\infty$, such that
\begin{equation}\label{eq:boundary10}
  \lim_{j\to\pm\infty}
  \left.\int(\psi\left|\partial_tf\right|^2
    +\lambda\phi''|\nabla f|^2+\lambda^3|f|^2\nabla\phi D^2\phi\nabla\phi)\,dx
    \right\vert_{t=T_j}=0.
\end{equation}

Our next step is to prove that in fact we have
\begin{equation*}
   \int_{-\infty}^{+\infty}\int(\psi\left|\partial_tf\right|^2
    +\lambda\phi''|\nabla f|^2+\lambda^3|f|^2\nabla\phi D^2\phi\nabla\phi)\,dx\,dt
    =0.
\end{equation*}
In order to this, we need to apply again Proposition \ref{prop:1}, for a suitable localization of $f$, and use what we have already proved. To this aim, let $R>0$ and $\chi_R\in\mathcal C^{\infty}(\R^n)$ be such that
\begin{equation*}
  \left.\chi_R\right\vert_{|x|\leq R}\equiv1,
  \qquad
  \left.\chi_R\right\vert_{|x|\geq2R}\equiv0,
  \qquad
  \left|\nabla\chi_R\right|\leq\frac2R,
  \qquad
  |\triangle\chi_R|\leq\frac2{R^2}.
\end{equation*}
Moreover, denote by $f_R:=\chi_Rf=\chi_Re^{\lambda\phi}u$, and notice by \eqref{eq:feq} that $f_R$ satisfies
\begin{align}\label{eq:fr10}
\left|\left(\partial_t^2-(\mathcal S+\mathcal A)\right)f_R\right|\leq
&
|Vf_R|+
\lambda\left|\left(\widetilde{W}\cdot\nabla\phi\right)f_R\right|+\left|W\cdot(\partial_t,\nabla)f_R\right|
\\
&
+\left|f\left(\triangle\chi_R-2\lambda\nabla\phi\cdot\nabla\chi_R-
\widetilde W\cdot\nabla\chi_R\right)\right|+2\left|\nabla\chi_R\cdot\nabla f\right|,
\nonumber
\end{align}
where $\mathcal S$ and $\mathcal A$ are given by \eqref{eq:SA}. We first prove a fundamental lemma about the boundary terms in \ref{prop:1}.
\begin{lemma}\label{lem:boundarynew}
  Let $R>1$, and $T_j$ be the sequence in \eqref{eq:boundary10}. Then
  \begin{equation}\label{eq:boundarynew}
    \lim_{j\to\pm\infty}\left|\langle \mathcal Af_R(T_j),\partial_tf_R(T_j)\rangle+
    \langle\psi f_R(T_j),\partial_tf_R(T_j)\rangle\right|=0.
  \end{equation}
\end{lemma}
\begin{proof}
  By \eqref{eq:SA} and Cauchy-Schwartz, we have 
  \begin{align*}
  &
    \left|\langle \mathcal Af_R(T_j),\partial_tf_R(T_j)\rangle+
    \langle\psi f_R(T_j),\partial_tf_R(T_j)\rangle\right|
    \\
    & \leq
    \lambda\|\nabla\phi\|_{L^\infty}\int\left(\left|\partial_tf_R\right|^2+|\nabla f_R|^2\right)\,dx
    +\frac\lambda2\|\triangle\phi\|_{L^\infty}\int(\left|f_R\right|^2+\left|\partial_tf_R\right|^2)\,dx
    \\
    & \ \ \ 
    +\|\psi\|_{L^\infty}\int(\left|f_R\right|^2+\left|\partial_tf_R\right|^2)\,dx.
  \end{align*}
  Hence, by \eqref{eq:fieps4}, \eqref{eq:fieps5}, \eqref{eq:psi} and the Poincar\'e inequality we obtain
  \begin{align*}
  &
    \left|\langle \mathcal Af_R(t),\partial_tf_R(t)\rangle+
    \langle\psi f_R(t),\partial_tf_R(t)\rangle\right|
    \\
    & \leq
    C\lambda\int_{|x|\leq2R}\left(\left|\partial_tf\right|^2+|\nabla f|^2+|f|^2\right)\,dx
    \\
    & 
    \leq 
    C\lambda\int_{|x|\leq2R}\left(\left|\partial_tf\right|^2+|\nabla f|^2\right)\,dx
    +
    C\lambda R^2\int_{|x|\leq2R}|\nabla f|^2\,dx
    \\
    & 
    \leq\sup_{|x|\leq2R}\frac{C\lambda}{\psi}\int\psi\left|\partial_tf\right|^2\,dx+
    \sup_{|x|\leq2R}\frac{C\lambda}{\phi''}(1+R^2)\int
    \phi''|\nabla f|^2\,dx
    \\
    &
    \leq\widetilde C\int\left(\psi\left|\partial_tf\right|^2+\phi''|\nabla f|^2\right)\,dx,
  \end{align*}
  for some $C,\widetilde C>0$. The proof now follows by \eqref{eq:boundary10}, once computing in $t=T_j$.
\end{proof}
\begin{remark}
  Notice that, at this level, we cannot prove \eqref{eq:boundarynew} for $f$ instead of $f_R$. Indeed, 
  \begin{equation*}
  \sup_{x\in\R^n}\frac1{\psi}=
  \sup_{x\in\R^n}\frac1{\phi''}=
  \infty.
  \end{equation*}
\end{remark}

We are now ready to conclude the proof of Theorem \ref{thm:wave}. We can now apply again Proposition \ref{prop:1} in \eqref{eq:fr10}. Notice that by \eqref{eq:assW} we have
\begin{equation*}
\left|\widetilde W\cdot\nabla\chi_R\right|\leq\frac{C}{R^{\frac32-}}\leq\frac CR,
\qquad
(R>1)
\end{equation*} 
for some constant $C>0$. Hence,  estimating as in Lemma \ref{lem:step2}, we get
\begin{align*}
    &
    \int_{T_j}^{T_k}\int \psi\left|\partial_tf_R\right|^2\,dx\,dt+
    \lambda\int_{T_j}^{T_k}\int \phi''|\nabla f_R|^2\,dx\,dt
    +
    \lambda^3\int_{T_j}^{T_k}\int\left|f_R\right|^2\nabla\phi D^2\phi\nabla\phi\,dx\,dt
    \\
    &
    \leq
    \frac{C\lambda}{R^2}\int_{T_j}^{T_k}\int_{R\leq|x|\leq2R}|f|^2\,dx\,dt+
    \frac{C}{R^2}\int_{T_j}^{T_k}\int_{R\leq|x|\leq2R}|\nabla f|^2\,dx\,dt
    \\
    & \ \ \ 
        +
    C\left.\left|\langle Af_R,\partial_tf_R\rangle+\langle\psi f_R,\partial_tf_R\rangle\right|\,
    \right\vert_{t=T_j}^{t=T_k},
  \end{align*}
for $R>1$ and $\lambda$ large enough, depending only on $V$ and $W$. Here we choose $T_j,T_k$ to be times of the sequence given by \eqref{eq:boundarynew}. Now we take the limits $j\to-\infty$, $k\to+\infty$ and obtain, by the previous estimate and \eqref{eq:boundarynew}, that
\begin{align}\label{eq:controlR}
    &
    \int_{-\infty}^{+\infty}\int_{|x|\leq2R}\psi\left|\partial_tf\right|^2\,dx\,dt+
    \lambda\int_{-\infty}^{+\infty}\int_{|x|\leq2R}\phi''|\nabla f|^2\,dx\,dt
    \\
    & \ \ \ 
    +
    \lambda^3\int_{-\infty}^{+\infty}\int_{|x|\leq2R}|f|^2\nabla\phi D^2\phi\nabla\phi\,dx\,dt
    \nonumber
    \\
    &
    \leq
    \frac{C\lambda^2}{R^2}\int_{-\infty}^{+\infty}\int_{R\leq|x|\leq2R}|f|^2\,dx\,dt+
    \frac{C}{R^2}\int_{-\infty}^{+\infty}\int_{R\leq|x|\leq2R}|\nabla f|^2\,dx\,dt,
    \nonumber
  \end{align}
for any $R>1$, and some $C>0$. By \eqref{eq:fieps5} and \eqref{eq:fi8},
\begin{equation*}
  \sup_{R\leq|x|\leq2R}\frac{1}{\nabla\phi D^2\phi\nabla\phi}=C_1R
  \left(1+\log\frac R\epsilon\right)^2,
  \sup_{R\leq|x|\leq2R}\frac{1}{\phi''}=C_2R\left(1+\log\frac R\epsilon\right)^2,
\end{equation*}
for some constants $C_1$ and $C_2>0$. Consequently, \eqref{eq:controlR} implies
\begin{align}\label{eq:controlRRR}
    &
    \int_{-\infty}^{+\infty}\int_{|x|\leq2R}\psi\left|\partial_tf\right|^2\,dx\,dt+
    \lambda\int_{-\infty}^{+\infty}\int_{|x|\leq2R}\phi''|\nabla f|^2\,dx\,dt
    \\
    & \ \ \ 
    +
    \lambda^3\int_{-\infty}^{+\infty}\int_{|x|\leq2R}|f|^2\nabla\phi D^2\phi\nabla\phi\,dx\,dt
    \nonumber
    \\
    &
    \leq
    \frac{C\lambda^2\left(1+\log\frac R\epsilon\right)^2}{R}
    \int_{-\infty}^{+\infty}\int_{R\leq|x|\leq2R}|f|^2\nabla\phi D^2\phi\nabla\phi\,dx\,dt
    \nonumber
    \\
    & \ \ \ +
    \frac{C\left(1+\log\frac R\epsilon\right)^2}{R}\int_{-\infty}^{+\infty}
    \int_{R\leq|x|\leq2R}\phi''|\nabla f|^2\,dx\,dt,
    \nonumber
  \end{align}
for any $R>1$. The two integrals at the right-hand side of \eqref{eq:controlRRR} are finite, by \eqref{eq:step2}, for any $R$: in conclusion, taking the limit as $R\to\infty$ in 
\eqref{eq:controlRRR} yields
\begin{equation*}
  \int\int\psi\left|\partial_tf\right|^2\,dx\,dt+
    \lambda\int\int\phi''|\nabla f|^2\,dx\,dt
    +
    \lambda^3\int\int|f|^2\nabla\phi D^2\phi\nabla\phi\,dx\,dt
    =0,
\end{equation*}
which implies that $f\equiv0$, and concludes the proof.

\appendix

\section{Repulsive potentials}\label{sec:app1}

We now sketch the proof of the more general result which is mentioned in Remark \ref{rem:3} above. The scheme of the proof is in fact completely analogous to the previous one. Notice that, if $u$ satisfy \eqref{eq:wavegen}, then denoting again by $f=e^{\lambda\phi}u$ we see that $f$ satisfy \eqref{eq:feq}, with $\widetilde W$ as in \eqref{eq:wtilde} and now
\begin{equation}\label{eq:SAgen}
\mathcal S=-\triangle+V_2-\lambda^2|\nabla\phi|^2,
\qquad
\mathcal A=\lambda\left(2\nabla\phi\cdot\nabla+\triangle\phi\right).
\end{equation}
In particular, the commutator is now given by
\begin{align}\label{eq:commutatorgen}
  [\mathcal S,\mathcal A] & =-\lambda\left(4\nabla\cdot D^2\phi\nabla-4\lambda^2\nabla\phi D^2\phi\nabla\phi
  +\triangle^2\phi\right)+2\lambda\nabla\phi\cdot\nabla V_2
  \\
  & =
  -\lambda\left(4\nabla\cdot D^2\phi\nabla-4\lambda^2\nabla\phi D^2\phi\nabla\phi
  +\triangle^2\phi\right) +2\lambda\phi'\left(\partial_rV_2\right)_+
  -2\lambda\phi'\left(\partial_rV_2\right)_-,
  \nonumber
\end{align}
since $\phi$ is radial. Now the only difference with the previous proof appears in Lemma \ref{lem:positivity}, in which we need to control the additional negative term containing $\left(\partial_rV_2\right)_-$ in the previous formula for the commutator $[S,A]$. It is immediate to see that, thanks to \eqref{eq:fieps4}, \eqref{eq:fi8} and assumption \eqref{eq:assV2}, we have
\begin{equation*}
  4\lambda^3\int\nabla\phi D^2\phi\nabla\phi|f|^2\,dx-2\lambda
  \int\phi'\left(\partial_rV_2\right)_-|f|^2\,dx
  \geq C\lambda^3\int\nabla\phi D^2\phi\nabla\phi|f|^2\,dx,
\end{equation*}
for some $C>0$ and $\lambda>\lambda_0$ sufficiently large, depending on 
$\left(\partial_rV_2\right)_-$. From here, the proof follows as in the previous case.

\section{Proof of Theorem \ref{thm:wave2}}\label{sec:app2}

Also in this case, the proof is analogous to the one of Theorem \ref{thm:wave}. The main difference here is in the choice of the multipliers. Let $\alpha\in\left[0,\frac12\right)$ and $p=\frac{4-2\alpha}{3}
\in\left(1,\frac43\right]$. In this case, we need to use a multiplier $\phi\in\mathcal C^4$ with the following properties:
\begin{equation}\label{eq:fieps4gen}
  \phi'(r)\sim
  \begin{cases}
    r
    \quad
    r\in[0,1]
    \\
    r^{p-1}
    \quad
    r>1,
  \end{cases}
\end{equation}
\begin{equation}\label{eq:fieps5gen}
  \phi''(r)=
  \begin{cases}
    C
    \quad
    r\in[0,1]
    \\
    r^{p-2}
    \quad
    r>1,
  \end{cases}
\end{equation}
\begin{equation}\label{eq:strigen}
  D^2\phi(r)\geq p(p-1)r^{p-2}Id.
\end{equation}
Moreover, by formula \eqref{eq:fi6} we deduce that
\begin{equation}\label{eq:fi8gen}
  \nabla\phi D^2\phi\nabla\phi > 
  \begin{cases}
    r^2
    \quad
    r\in[0,\epsilon]
    \\
     \frac{C}{r^{4-3p}}
     \quad
     r>\epsilon,
  \end{cases}
\end{equation}
\begin{equation}\label{eq:fi9gen}
  \triangle^2\phi\in L^\infty(\R^n),
  \qquad
  \left|\triangle^2\phi\right|\leq\frac{C}{r^{p-4}}
  \quad
  (r>1).
\end{equation}
For the existence of such a function we refer the reader to the Appendix 7 part (a) in \cite{EKPV2}.
Moreover we also define 
\begin{equation}\label{eq:psigen}
  \psi(r)=\delta\lambda\phi''(r),
\end{equation}
in analogy with \eqref{eq:psi},
for some constant $\delta>0$ to be chosen sufficiently small.
One can check now that all the steps of the previous proof for Theorem \ref{thm:wave} work exactly in the same way with these new multipliers and under assumptions \eqref{eq:assV00}, 
\eqref{eq:assW00} and \eqref{eq:assumptionH00}. We omit further details (See the proof of Theorem 1 in \cite{EKPV2}).

\end{document}